\documentclass[11pt, a4paper]{amsart}
\usepackage{a4wide}

\usepackage{amssymb}
\usepackage{amsthm}
\usepackage{amsfonts}

\newtheorem{lemma}{Lemma}[section]
\newtheorem{theorem}[lemma]{Theorem}
\newtheorem{corollary}[lemma]{Corollary}
\newtheorem{proposition}[lemma]{Proposition}

\newtheorem{example}[lemma]{Example}

\newtheorem{problem}[lemma]{Problem}

\theoremstyle{definition}
\newtheorem{definition}[lemma]{Definition}
\newtheorem{remark}[lemma]{Remark}

\newcommand{\w}{\omega}

\newcommand{\F}{\mathcal F}
\newcommand{\IN}{\mathbb N}
\newcommand{\E}{\mathcal E}
\newcommand{\C}{\mathcal C}
\newcommand{\M}{\mathcal M}
\newcommand{\Ra}{\Rightarrow}

\begin{document}

\title{Minimal covers of hypergraphs}

\author{Taras Banakh}
\address{Ivan Franko National University of Lviv,
Lviv (Ukraine) and\newline Jan Kochanowski University in Kielce (Poland).}
\email{\tt t.o.banakh@gmail.com}
\author{Dominic van der Zypen}
\address{Swiss Armed Forces Command Support, CH-3003 Bern,
Switzerland}
\email{\tt dominic.zypen@gmail.com}

\subjclass[2010]{05C65, 05C70, 54D20}

\keywords{Hypergraph, cover, minimal cover, minicompact space}

\begin{abstract} For a hypergraph $H=(V,\E)$, a subfamily $\C\subseteq \E$ is called a cover of the hypergraph if $\bigcup\C=\bigcup\E$. A cover $\C$ is called minimal if each cover $\mathcal D\subseteq\C$ of the hypergraph $H$ coincides with $\C$. We prove that for a hypergraph $H$ the following conditions are equivalent: (i) each countable subhypergraph of $H$ has a minimal cover; (ii) each non-empty subhypergraph of $H$ has a maximal edge; (iii) $H$ contains no isomorphic copy of the hypergraph $(\w,\w)$. This characterization implies that a countable hypergraph $(V,\E)$ has a minimal cover if every infinite set $I\subseteq V$ contains a finite subset $F\subseteq I$ such that the family of edges $\E_F:=\{E\in\E:F\subseteq E\}$ is finite. Also we prove that a hypergraph $(V,\E)$ has a minimal cover if $\sup\{|E|:E\in\E\}<\w$ or for every $v\in V$ the family $\E_v:=\{E\in\E:v\in E\}$ is finite.
\end{abstract}

\maketitle
\parskip=1pt
\section{Introduction}
Hypergraphs are very simple mathematical structures, but they are 
surprisingly useful for modeling various concepts in the real
world. 

A {\em hypergraph} is a pair $\Gamma=(V,\E)$ consisting of a set $V$ of {\em vertices} and a collection $\E$ of subsets of $V$, called the {\em edges} of $\Gamma$. In this paper we shall be interested in hypergraphs $(V,\E)$ whose edges cover $V$ in the sense that $V=\bigcup \E:=\bigcup_{E\in\E}E$. In this case the hypergraph can be identified with the family $\E$ of its edges. Moreover, any family $\E$ of sets can be identified with the hypergraph $(\bigcup \E,\E)$. 
To shorten notations, by a {\em hypergraph} we shall understand any family $\E$ of sets. Elements of $\E$ and $\bigcup \E$ will be called the {\em edges} and the {\em vertices} of the hypergraph $\E$, respectively. 
A hypergraph $\E$ is {\em countable} if the sets $\E$ and $\bigcup\E$ both are countable.

Two hypergraphs $\E$ and $\E'$ are called {\em isomorphic} if there exists a bijective map $f:\bigcup\E\to\bigcup\E'$ such that $\mathcal E'=\{f[E]:E\in\E\}$.

To distinguish between (pre)images of points and sets, for a function $f:X\to Y$ between sets $X,Y$ and subsets $A\subseteq X$ and $B\subseteq Y$, we put $$f[A]:=\{f(a):a\in A\}\mbox{ and }f^{-1}[B]:=\{x\in X:f(x)\in B\}.$$

For a hypergraph $\E$ and a set $A$ put $\E{\restriction}A:=\{E\cap A:E\in\E\}$. A hypergraph $\mathcal S$ is called a {\em subhypergraph} of a hypergraph $\E$ if $\mathcal S=\mathcal F{\restriction}A$ for some sets $\mathcal F\subseteq \E$ and $A\subseteq\bigcup\mathcal F$.

An edge $E\in\E$ of a hypergraph $\E$ is {\em maximal} if each edge $E'\in\E$ with $E\subseteq E'$ is equal to $E$.

A subset $\C\subseteq \E$ of a hypergraph $\E$ is called a {\em cover} of the hypergraph if $\bigcup \C=\bigcup \E$. A cover $\C$ of a hypergraph $\E$ is {\em minimal} if each cover $\mathcal D\subseteq \C$ of $\E$ coincides with $\C$. Observe that a cover $\C\subseteq \E$ of a hypergraph $\E$ is minimal if and only if  every edge $C\in \C$ contains a vertex $v\in C$ such that the set $\C_v:=\{E\in \C:v\in E\}$ coincides with the singleton $\{C\}$.

It is clear that a hypergraph $\E$ has a minimal cover if the set $\E$ is finite.
The simplest example of a hypergraph without minimal covers is the hypergraph $\w$ whose edges are finite ordinals $n$ identified with the sets $\{0,\dots,n-1\}$ of smaller ordinals. Observe that $\bigcup\w=\w$, which means that in the hypergraph  $\w$ the sets of vertices and edges coincide. 

The main result of this paper is the following characterization.

\begin{theorem}\label{t:main} For any hypergraph $\E$ the following conditions are equivalent:
\begin{enumerate}
\item each countable subhypergraph of $\E$ has a minimal cover;
\item each non-empty subhypergraph of $\E$ has a maximal edge;
\item no subhypergraph of $\E$ is isomorphic to $\w$.
\end{enumerate}
\end{theorem}

The implications $(1)\Ra(3)\Leftrightarrow(2)$ of this theorem are trivial and $(3)\Ra(1)$ will be proved in Lemma~\ref{l3}. In Section~\ref{s:a} we shall analyze some implications of Theorem~\ref{t:main} and pose related open problems.  It should be mentioned that problems related to minimal covers of finite hypergraphs have been intensively studied in Graph Theory, see \cite{AB}, \cite{Lehel}, \cite{Okun}, \cite{TD}.

\section{Key Lemmas}
 
 Theorem~\ref{t:main} will be derived from Lemma~\ref{l3}, proved in this section. In the proof of Lemma~\ref{l3} we shall one two other lemmas.

\begin{lemma}\label{l:delete} If a hypergraph $\E$ on the set $V:=\bigcup\E$ has no minimal covers, then for any edge $F\in\mathcal E$ the hypergraph $\E{\restriction}V\setminus F$ has no minimal covers, too.
\end{lemma}

\begin{proof} Assuming that the hypergraph $\E{\restriction}V\setminus F$ has a minimal cover, we can find a minimal subfamily $\M\subseteq \E$ such that $V\setminus F=\bigcup_{E\in\M}E\setminus F$. If $\bigcup\M=V$, then $\M$ is a minimal cover of $\E$.
If $\bigcup\M\ne V$, then $\{F\}\cup\M$ is a minimal cover of $\bigcup\E$.
\end{proof}
 
To formulate our next lemma we need to introduce some notation. For a hypergraph $\E$ and subsets $A,B\subseteq \bigcup \E$ let $$\E_A:=\{E\in\E:A\subseteq E\},\;\E_{-B}:=\{E\in\E:E\cap B=\emptyset\}\mbox{ and }\E_{A-B}:=\E_A\cap\E_{-B}.$$
If $A=\{v\}$ for some vertex $v\in V$, then we shall write $\E_v$ and $\E_{v-B}$ instead of $\E_{\{v\}}$ and $\E_{\{v\}-B}$, respectively.

\begin{definition} A hypergraph $\E$ is defined to have {\em minimal covers at a vertex} $v\in\bigcup\E$ if for any finite set $B\subseteq \bigcup\E$ and any subset $A\subseteq \bigcup\E_{v-B}$ the hypergraph $\E_{v-B}{\restriction}A\setminus\bigcup\E_{-(B\cup\{v\})}$ has a minimal cover.  
\end{definition}

\begin{lemma}\label{l:local} A countable hypergraph $\E$ has a minimal cover if  $\E$ has minimal covers at each vertex $v\in \bigcup\E$.
\end{lemma}

\begin{proof} Fix a well-order $\le$ of the countable set $V:=\bigcup\E$ such that for any vertex $v\in V$ the set $\{u\in V:u<v\}$ has finite cardinality $<|V|$. For a non-empty subset $S\subseteq V$, we denote by $\min S$ the smallest element of $S$ with respect to the well-order $\le$. 

Let $V_0:=V$ and $\M_0=\emptyset$.
By induction, for every $n\in\w$ we shall choose a subfamily $\M_n\subseteq \E$ such that the following conditions are satisfied:
\begin{enumerate}
\item[$(1_n)$] if the set $V_n:=V\setminus\bigcup_{k\le n}\bigcup \M_{k}$ is empty, then $\M_{n+1}=\emptyset$;
\item[$(2_n)$] if $V_n\ne\emptyset$, then 
\begin{itemize}
\item $V_n\subseteq \bigcup\E_{-B_{<n}}$ where $B_{<n}:=\{\min V_k\}_{k<n}$;
\item $\M_{n+1}\subseteq\E_{v_n-B_{<n}}$ where $v_n:=\min V_n$;
\item $\M_{n+1}{\restriction}V_n\setminus \bigcup\E_{-B_{\le n}}$ is a minimal cover of the hypergraph
$\E_{v_n -B_{<n}}{\restriction}V_n\setminus \bigcup\E_{-B_{\le n}}$ where $B_{\le n}:=\{v_k\}_{k\le n}$. 
\end{itemize}
\end{enumerate}

Assume that for some $n\in\w$ we have chosen families $\M_0,\dots,\M_{n}$ such that the conditions $(1_{k}),(2_k)$ are satisfied for all $k<n$.
Now we choose a family $\M_{n+1}$ satisfying the conditions $(1_n)$ and $(2_n)$. If the set $V_n:=V\setminus \bigcup_{k\le n}\bigcup \M_{k}$ is empty, then put $\M_{n+1}=\emptyset$. 

Now assume that the set $V_n$ is not empty and let $v_n:=\min V_n$ be the smallest element of $V_n$ with respect to the well-order $\le$ on $V$. Since the sequence $V_0\supseteq\dots\supseteq V_n$ consists of non-empty sets, the sets $B_{<n}=\{\min V_k\}_{k<n}$  and $B_{\le n}=\{\min V_k\}_{k\le n}$ are well-defined.  

First we check that $V_n\subseteq\bigcup\E_{-B_{<n}}$. Observe that 
$V_n=V_{n-1}\setminus\bigcup\M_n$. Then for $k=n-1$ the condition $(2_k)$ ensures that any vertex $x\in V_n$ belongs to the set $V_{k}\subseteq \bigcup\E_{B_{<k}}$ but not to the set $\bigcup\M_n\supseteq (\bigcup\E_{v_k-B_{<k}})\cap V_k\setminus\bigcup \E_{-B_{\le k}}$. Then $x\notin  (\bigcup\E_{v_k-B_{<k}})\setminus\bigcup \E_{-B_{<n}}$. Assuming that $x\notin \bigcup\E_{-B_{<n}}$, we conclude that $x\notin  \bigcup\E_{v_k-B_{<k}}$. Since $x\in V_n\subseteq V_k\subseteq \bigcup\E_{-B_{<k}}$, we can find an edge $E\in\E_{-B_{<k}}$ containing $x$. Since  $x\notin  \bigcup\E_{v_k-B_{<k}}$, the edge $E$ does not contain the point $v_k=v_{n-1}$. Then $x\in E\in \E_{-B_{<n}}$, which contraducts our assumption $x\notin \bigcup\E_{-B_{<n}}$.
This contradiction completes the proof of the inclusion   $V_n\subseteq\bigcup\E_{-B_{<n}}$.

Since $\E$ has minimal covers at the vertex $v_n$, the hypergraph $\E_{v_n-B_{<n}}{\restriction}V_n\setminus\bigcup\E_{-B_{\le n}}$ has a minimal cover $\M'_{n+1}$.
For every edge $E\in \M'_{n+1}$ choose an edge $\tilde E\in \E_{v_n-B_{<n}}$ such that $E=\tilde E\cap V_n\setminus \bigcup \E_{-B_{\le n}}$ and put $\M_{n+1}:=\{\tilde E:E\in \M'_{n+1}\}$. This completes the inductive step.

After completing the inductive construction, consider the family $\M:=\bigcup_{n\in\w}\M_n\subseteq\E$. We claim that $\M$  is a minimal cover of the hypergraph $\E$. First we show that $\bigcup\M=V$. Assuming that some vertex $v\in V$ does not belong to  $\bigcup\M$, we conclude that $v\in\bigcap_{n\in\w}V_n$ and hence the set ${\downarrow}v\supseteq\{v_n\}_{n\in\w}$ is infinite, which contradicts the choice of the well-order $\le$ (here we should also observe that for every $n\in\w$ we have $v_n\in V_n\subseteq \bigcup\E_{-B_{<n}}$, which implies that $v_n\notin B_{<n}$ and hence the points $v_n$, $n\in\w$, are pairwise distinct).

Next, we show that the cover $\M$ of $\E$ is minimal. Given any edge $E\in\M$, find the smallest number $n\in\w$ such that $E\in \M_{n+1}$. By the inductive condition $(2_{n})$, the family $\M_{n+1}{\restriction}V_n\setminus\bigcup\E_{-B_{\le n}}$ is a minimal cover of the hypergraph $\E_{v_n-B_{<n}}{\restriction}V_n\setminus\bigcup\E_{-B_{\le n}}$. Consequently, the edge $E$ contains a point $x\in (\bigcup\E_{v_n-B_{<n}})\cap V_n\setminus \bigcup\E_{-B_{\le n}}$, which is not contained in any edge $E'\in\M_{n+1}\setminus\{E\}$. 

 The definition $(1_n)$ of the set $V_n\ni x$ implies that $x\notin E'$ for any edge $E'\in\bigcup_{k\le n}\M_k$. On the other hand, the inclusion $\M_{m+1}\subseteq \E_{-B_{<m}}\subseteq \E_{-B_{\le n}}$ holding for every $m>n$ ensures that the point $x\notin \bigcup\E_{-B_{\le n}}$ does not belong to any edge $E'\in \bigcup_{m>n}\M_{m+1}$. This completes the proof of the minimality of the cover $\M$ and also the proof of the lemma.
\end{proof}

The following Lemma proves the (only non-trivial) implication $(3)\Ra(1)$ of Theorem~\ref{t:main}.

\begin{lemma}\label{l3} If a countable hypergraph $\E$ has no minimal covers, then  $\E$  contains a subhypergraph isomorphic to $\w$.
\end{lemma}

\begin{proof} 
Applying Lemma~\ref{l:local}, we can find a vertex $v_0\in V$ at which the hypergraph $\E$ has no minimal covers. This means that for some finite set $B_0\subseteq V$ and some set $A_0\subseteq V$ the hypergraph $\E_{v_0-B_0}{\restriction}A_0\setminus\bigcup\E_{-(B_0\cup\{v_0\})}$ has no minimal covers.
This implies that the hypergraph $\bigcup\E_{v_0-B_0}$ is not empty, so there exists an edge $E_0\in \E_{v_0-B_0}$. Finally, put $V_0:=(\bigcup\E_{v_0-B_0})\cap A_0\setminus(\{E_0\}\cup\bigcup\E_{-(B_0\cup\{v_0\})})$.
By Lemma~\ref{l:delete}, the hypergraph $\E_{v_0-B_0}{\restriction}V_0$ has no minimal covers. 

Proceeding by induction, for every $n\in\w$ we shall choose subsets $V_n,A_n\subseteq V$,  a finite subset $B_n\subseteq V$, a point $v_n\in V$, and an edge $E_n\in\E$ such that the following conditions are satisfied:
\begin{enumerate}
\item[$(1_n)$] $v_n\in V_{n-1}$;
\item[$(2_n)$] $B_n\cup A_n\subseteq V_{n-1}$;
\item[$(3_n)$] for the sets $v_{<n}:=\{v_k\}_{k<n}$, $v_{\le n}:=\{v_k\}_{k\le n}$ and $B_{\le n}:=\bigcup_{k\le n}B_k$, the hypergraph $\E_{v_{\le n}-B_{\le n}}{\restriction}A_n\setminus\bigcup\E_{v_{<n}-(B_{\le n}\cup\{v_n\})}$ has no minimal covers;
\item[$(4_n)$] $v_n\in E_n\in \E_{v_{\le n}-B_{\le n}}$;
\item[$(5_n)$] $V_n:=(\bigcup\E_{v_{\le n}-B_{\le n}})\cap A_n\setminus (E_n\cup \bigcup\E_{v_{<n}-(B_{\le n}\cup\{v_n\})})$;
\item[$(6_n)$] the hypergraph $\E_{v_{\le n}-B_{\le n}}{\restriction}V_n$ has no minimal covers.
\end{enumerate}
Observe that for the number $n=0$ the inductive conditions $(1_n)$--$(6_n)$ are satisfied. Assume that for some number $n\in\IN$ and every $k<n$ we have chosen sets $V_k,A_k,B_k$, a point $v_k$ and an edge $E_k$ so that the inductive conditions 
$(1_k)$--$(6_k)$ are satisfied for all $k<n$.

By the condition $(6_{n-1})$ the  hypergraph $\E_{v_{<n}-B_{<n}}|V_{n-1}$ has no minimal covers. By Lemma~\ref{l:local}, this hypergraph does not have no minimal covers at some vertex $v_n\in V_{n-1}$. This means that for some set $A_n\subseteq V_{n-1}$ and some finite set $B_n\subseteq A_{n-1}$ the hypergraph 
$\E_{v_{\le n}-B_{\le n}}{\restriction}A_n\setminus \bigcup \E_{v_{<n}-(B_{\le n}\cup\{v_n\})}$ has no minimal covers. This implies that the family $\E_{v_{\le n}-B_{\le n}}$ is not empty, so we can find an edge $E_n\in \E_{v_{\le n}-B_{\le n}}$. By Lemma~\ref{l:delete}, for the set $V_n:=(\bigcup\E_{v_{\le n}-B_{\le n}})\cap A_n\setminus(\{E_n\}\cup\bigcup\E_{v_{<n}-(B_{\le n}\cup\{v_n\})})$ the hypergraph $\E_{v_{\le n}-B_{\le n}}|V_n$ has no minimal covers.
Now we see that the conditions $(1_n)$--$(6_n)$ are satisfied.
This completes the inductive step.

After completing the inductive construction, observe that the conditions $(1_n),(4_n),(5_n)$ ensure that the sequence $(v_n)_{n\in\w}$ consists of pairwise distinct points. Moreover, for the bijective map $f:\w\to \Omega:=\{v_n\}_{n\in\IN}$, $f:n\mapsto v_{n+1}$, we have $f^{-1}(E_n)=\{0,\dots,n-1\}$ for all $n\in\w$, which means that the subhypergraph $\{E_n\}_{n\in\w}{\restriction}\Omega$ of $\E$ is isomorphic to the graph $\w$.
\end{proof}

\section{On minimal covers in  hypergraphs}\label{s:a}

In this section we discuss the problem of existence of minimal covers in arbitrary (not necessarily countable) hypergraphs.
First, let us analyze some implications of Theorem~\ref{t:main}.

\begin{corollary}\label{c:fin} A countable hypergraph $\E$ has a minimal cover if 
 any infinite set $I\subseteq \bigcup\E$ contains a finite subset $F\subseteq I$ such that the family $\E_F:=\{E\in\E:F\subseteq E\}$ is finite.
\end{corollary}

\begin{proof} Assuming that $\E$ has no mimimal subcovers and applying Theorem~\ref{t:main}, we can find subsets $\mathcal I\subseteq\E$ and $I\subseteq\bigcup\mathcal I$ such that the subhypergraph $\mathcal I{\restriction}I$ is isomorphic to the hypergraph $\w$.  Then for any finite subset $F$ of the infinite set $I$ the family $\mathcal I_F$ is infinite and so is the family $\E_F\supseteq\mathcal I_F$.
\end{proof}

\begin{corollary}\label{c:nm} A countable hypergraph $\E$ has a minimal cover if there exist  numbers $n,m\in\IN$ such that for any $n$-element subfamily $\F\subseteq\E$ the intersection $\bigcap\F$ has cardinality $|\bigcap\F|<m$.
\end{corollary}

\begin{proof} Assume that for any $n$-element subfamily $\F\subseteq\E$ the intersection $\bigcap\F$ has cardinality $|\bigcap\F|<m$. Given any infinite set $I\subseteq \bigcup E$, choose any $m$-element set $F\subseteq I$ and observe that the family $\E_F$ has cardinality $|\E_F|<n$ (otherwise $|F|\le |\bigcap\E_F|<m$). By Corollary~\ref{c:fin}, the hypergraph $\E$ has a minimal cover.
\end{proof}

\begin{remark} For $m=2$ Corollary~\ref{c:nm} was essentially proved by the Mathoverflow user @bof in his answer to the problem \cite{MO1} posed by the first author.
\end{remark}

We do not know if Corollary~\ref{c:nm} can be generalized to arbitrary (not necessarily countable) hypergraphs. However this can be done if $n$ or $m$ is equal to 1.

\begin{proposition}\label{p:2} A hypergraph $\E$ has a minimal cover if of if the following conditions holds:
\begin{enumerate}
\item for any $v\in\bigcup\E$ the family $\E_v:=\{E\in\E:v\in E\}$ is finite;
\item $\sup\{|E|:E\in\E\}<\w$.
\end{enumerate}
\end{proposition}

\begin{proof} 1. First assume that for any $v\in V:=\bigcup\E$ the family $\E_v:=\{E\in\E:v\in E\}$ is finite.

By Zorn's Lemma the hypergraph $\E$ contains a maximal subfamily $\M\subseteq \E$ such that $|\M_v|<|\E_v|$ for all $v\in V$. We claim that $\C:=\E\setminus \M$ is a minimal cover of $\E$. Indeed, for every $v\in V$ the inequality $|\M_v|<|\E_v|$ implies that $\bigcup \C=V$. By the maximality of $\M$, for every $E\in \C$ and the family $\M'=\M\cup\{E\}$, there exists a vertex $v\in V$ such that $E_v=M'_v$, which means that $v\notin \bigcup (\C\setminus\{E\})$ and $\C\setminus\{E\}$ is not a cover of $V$.
\smallskip

2. By induction for every $m\in\w$ we shall prove that any hypergraph $\E$ with $\sup\{|E|:E\in \E\}\le m$ has a minimal cover.

For $m=1$, the 
assertion is trivial. Assume that for some $m\ge 2$ we have proved that any 
hypergraph $\E$ such that $\sup\{|E|:E\in\E\}<m$ has a minimal cover.

Take any hypergraph $\E$ with $\sup\{|E|:E\in\E\}\le m$. Using Zorn's lemma, 
choose a maximal disjoint subfamily $\mathcal D\subseteq \E$. By the maximality of $\mathcal D$, each edge 
$E\in \mathcal E\setminus \mathcal D$ intersects the set $\bigcup \mathcal D$, which implies that the hypergraph $\E':=\{E\setminus\bigcup\mathcal D:E\in\E\}$ has $\sup\{|E'|:E'\in\E'\}<\sup\{|E|:E\in\E\}\le m$. By the inductive assumption, 
the hypergraph  $\E'$ has a minimal cover $\C\subseteq \E'$. For every edge $C\in \C$ find an edge $\tilde C\in \E$ such that 
$C=\tilde C\setminus\bigcup\mathcal D$. Let $\tilde{\C}:=\{\tilde C:C\in \C\}$ and $\tilde{\mathcal D}:=\{D\in\mathcal D:D\not\subseteq \bigcup\tilde{\C}\}$. It can be shown that $\tilde C\cup\tilde{\mathcal D}$ is a minimal cover of the hypergraph $\E$.
\end{proof}

\begin{problem}\label{prob:nm} Assume that for a hypergraph $\E$ there exist  numbers $n,m\in\IN$ such that for any $n$-element subfamily $\F\subseteq\E$ the intersection $\bigcap\F$ has cardinality $|\bigcap\F|<m$. Has $\E$ a minimal cover?
\end{problem}

The answer to this problem is not known even for $n=m=2$. The following special (but spectacular) case of Problem~\ref{prob:nm} is also open, see  \cite{MO1}.

\begin{problem}[Banakh, 2018] Let $\E$ be a cover of the real plane by lines. Has $\E$ a minimal subcover?
\end{problem}

Finally, let us consider an example of a hypergraph consisting of pairwise incomparable sets and having no minimal cover. This example was first presented  by P.D\"om\"ot\"or in his answer to a question of the second author on Mathoverflow \cite{MO2} (see also the discussion at \cite{MO3}).

\begin{example}[D\"om\"ot\"or, 2015] The cover $$\E:=\big\{[-n,0]\cup\{n\},\{-n\}\cup[0,n]:n\ge 2\big\}$$of $\mathbb Z$ contains no minimal subcovers and consists of pairwise incomparable finite sets.
\end{example}



\begin{thebibliography}{99}



\bibitem{AB} R.~Aharoni, E.~Berger, {\it Menger's theorem for infinite graphs}, 
	Inventiones Math. {\bf 176}:1 (2009), 1--62.

\bibitem{MO1} T.Banakh, {\em Is each cover of the plane by lines minimizable?}, \newline({\tt https://mathoverflow.net/questions/308761/is-each-cover-of-the-plane-by-lines-minimizable}).



\bibitem{Jech} T.~Jech, {\em Set Theory}, Springer, 2003.

\bibitem{Lehel} J.~Lehel, {\em Covers in hypergraphs}, Combinatorica {\bf 2}:3 (1982), 305--309. 

\bibitem{Okun} M.~Okun, {\em On approximation of the vertex cover problem in hypergraphs}, Discrete Optim. {\bf 2}:1 (2005), 101--111.

\bibitem{TD} D.K.~Thakkar, V.R.~Dave, {\em Edge Cover in a Hypergraph}, Intern. J. Math. and Appl. {\bf 5}:4-E (2017), 761--768.

\bibitem{MO2} D. van der Zypen, {\em Strongly minimal covers},\newline ({\tt https://mathoverflow.net/questions/193352/strongly-minimal-covers}).

\bibitem{MO3} D. van der Zypen, {\em Minimal covers in hypergraphs with finite edges},\newline ({\tt https://mathoverflow.net/questions/308265/minimal-covers-in-hypergraphs-with-finite-edges}).



\end{thebibliography}
\end{document}